\documentclass[a4paper,12pt,reqno]{amsart}
\usepackage{amsfonts}
\usepackage{amsmath}
\usepackage{amssymb}
\usepackage[a4paper]{geometry}
\usepackage{mathrsfs}
\usepackage[colorlinks]{hyperref}

\usepackage[all]{xy}
\usepackage[utf8]{inputenc}
\usepackage{slashed}
\usepackage{cleveref}
\usepackage{amsmath}
\usepackage{mathtools}
\usepackage{enumerate}
\usepackage{tgadventor}
\usepackage[
backend=biber,
style=alphabetic,
sorting=ynt
]{biblatex}
\newcommand{\beq}[1]{\begin{equation}\label{eq:#1}}
	\newcommand{\eeq}{\end{equation}}

\newtheorem{theorem}{Theorem}[section]
\newtheorem{lemma}[theorem]{Lemma}
\newtheorem{notation}[theorem]{Notation}

\newtheorem{proposition}[theorem]{Proposition}
\newtheorem{corollary}[theorem]{Corollary}

\theoremstyle{definition}
\newtheorem{definition}[theorem]{Definition}
\newtheorem{example}[theorem]{Example}

\theoremstyle{remark}
\newtheorem{remark}[theorem]{Remark}
\numberwithin{equation}{section}

\def\G{{\mathbb{G}}}
\def\R{{\mathcal{R}}}
\def\L{{^{2}_{L^2(\G)}}}

\newcommand{\vertiii}[1]{{\left\vert\kern-0.25ex\left\vert\kern-0.25ex\left\vert #1 
    \right\vert\kern-0.25ex\right\vert\kern-0.25ex\right\vert}}
\allowdisplaybreaks

\begin{document}

	\setcounter{page}{1}
	
\title[The heat equation with singular potentials]{The heat equation with singular potentials. II: Hypoelliptic case}

\author[M. Chatzakou]{Marianna Chatzakou}
\address{
	Marianna Chatzakou:
	\endgraf
    Department of Mathematics: Analysis, Logic and Discrete Mathematics
    \endgraf
    Ghent University, Belgium
  	\endgraf
	{\it E-mail address} {\rm marianna.chatzakou@ugent.be}
		}

\author[M. Ruzhansky]{Michael Ruzhansky}
\address{
  Michael Ruzhansky:
  \endgraf
  Department of Mathematics: Analysis, Logic and Discrete Mathematics
  \endgraf
  Ghent University, Belgium
  \endgraf
 and
  \endgraf
  School of Mathematical Sciences
  \endgraf
  Queen Mary University of London
  \endgraf
  United Kingdom
  \endgraf
  {\it E-mail address} {\rm michael.ruzhansky@ugent.be}
  }

\author[N. Tokmagambetov]{Niyaz Tokmagambetov}
\address{
  Niyaz Tokmagambetov:
  \endgraf
  Department of Mathematics: Analysis, Logic and Discrete Mathematics
  \endgraf
  Ghent University, Belgium
  \endgraf
  and
  \endgraf   
  Institute of Mathematics and Mathematical Modeling
  \endgraf
  Almaty, Kazakhstan
  \endgraf
  and
\endgraf   
Al--Farabi Kazakh National University
\endgraf
Almaty, Kazakhstan
\endgraf
{\it E-mail address} {\rm niyaz.tokmagambetov@ugent.be, tokmagambetov@math.kz}
  }

\thanks{The authors are supported by the FWO Odysseus 1 grant G.0H94.18N: Analysis and Partial Differential Equations and by the Methusalem programme of the Ghent University Special Research Fund (BOF) (Grant number 01M01021). Michael Ruzhansky is also supported by EPSRC grants EP/R003025/2 and EP/V005529. }

\begin{abstract} 
In this paper we consider the heat equation with a strongly singular potential and show that it has a very weak solution. Our analysis is devoted to general hypoelliptic operators and is developed in the setting of graded Lie groups. The current work continues and extends the work \cite{ARST21c}, where the classical heat equation on $\mathbb R^n$ was considered.  
\end{abstract}

\maketitle

\tableofcontents

\section{Introduction}
Baras and Goldstein  \cite{BG84a,BG84b} studied the heat equation with the singular potential in bounded and unbounded domains of $\mathbb{R}^n$ that initiated a series of further works on the topic; see \cite{ARST21c} for a more detailed exposition of them. Later on, Goldstein and Zhang in \cite{GZ01} studied the heat equation with singular potentials of certain type on the Heisenberg group $\mathbb{H}_n$, where the usual Laplace operator on $\mathbb{R}^n$ was replaced by the sub-Laplacian on $\mathbb{H}_n$.

Our consideration here is more broad not only in the sense that we allow for the setting to be any graded Lie group and the differential operator to be any positive Rockland operator, but we also allow for the potential to be a of any sign distributional function. Specifically, the following Cauchy problem is considered: 
\begin{equation}
\label{heat.eq}
\begin{cases}
u_{t}(t,x) +\mathcal{R}u(t,x)+V(x)u(t,x)=0\,,\quad (t,x)\in [0,T]\times \mathbb{G}\,,T>0\,,\\
u(0,x)=u_0(x)\,, x \in \G\,,	
\end{cases}       
\end{equation}
where $\G$ is a graded Lie group, $\R$ is any positive Rockland operator on $\G$, $V$ is a distributional function and $T>0$ is any real number. 

The potential $V$ is allowed to have $\delta$-type singularities. To conquer the problem of multiplication of distributions \cite{Sch54} we follow the theory of very weak solutions as introduced in \cite{GR15} where the theory initially applied on wave equations with irregular coefficients. To give an overview of the topic of very weak solutions we refer to the works \cite{RT17b} and \cite{MRT19}
on the analysis of the problems with singular, time-dependent coefficients on the Euclidean setting. The application of the concept of the very weak solutions in the setting of a graded Lie group, appears in \cite{RY20, CRT21, CRT21b}.

In this work we show that the notion of the very weak solution is applicable to the Cauchy problem \eqref{heat.eq}, and, when the classical solution exists, it converges to very weak one. The present work can be served as an extension of its first part \cite{ARST21c} addressed to the heat equation in the Euclidean setting; i.e., the consideration there expressed with the present terminology becomes $\G=\mathbb R^d$ and $\mathcal{R}=-\Delta$ being the positive fractional Laplacian on the Euclidean space. The novelty of this work,  that extends the realisation of the problem already in the Euclidean setting, relies  on two main considerations: we allow the potential $V$ in \eqref{heat.eq} to be of any sign, and extend the functional space of $V$ by making use of Sobolev embedding techniques. Finally, let us draw the reader's attention to Remarks \ref{rem.negl} and  \ref{finrem} which rectify or clarify the statements of the first consideration \cite{ARST21c} on the notion of the uniqueness of the very weak solution, and on its recapture of the classical one, respectively.

\section{Notions and terminology}
Let us briefly recall some, useful for our purposes, notions in the setting of graded groups. What follows is a shorter version of the introduction given in the recent papers by the same authors; see \cite{CRT21,CRT21b} for some other aspects of this material. Interested readers can consult the detailed exposition by Folland and Stein [Chapter 1 in \cite{FS82}]. A more recent comprehensive approach on the topic can be found in the work of Fischer and the second author; see [Chapter 3 in \cite{FR16}]. 

\subsection{Graded Lie groups} 

Let $\G$ be a connected simply connected Lie group. We call $\G$ a \textit{graded} Lie group if the next condition is satisfied for its Lie algebra $\mathfrak{g}$: 
\[ 
 \mathfrak{g}= \bigoplus_{i=1}^{\infty}  \mathfrak{g}_{i}\,,
\]
where all, but finitely many elements of the below vector space decomposition $\mathfrak{g}_{i}$'s, are equal to $\{0\}$. Additionally, the actions of each $\mathfrak{g}_{i}$ shall satisfy $[\mathfrak{g}_{i},\mathfrak{g}_{j}]\subset \mathfrak{g}_{i+j}$, where we have set $\mathfrak{g}_0:=\mathfrak{g}$.

Graded Lie groups are \textit{homogeneous} and their natural relevant structure is the so-called ``canonical'' dilations which appears also, due to the diffeomorphism of the two, as the structure of the corresponding Lie algebra $\mathfrak{g}$. If $\nu_{1},\cdots,\nu_{n}$ are the weights of the dilations, i.e., the powers appearing in the automorphism, denoted by $D_r$, $r>0$, that is expressed as 
\begin{equation}\label{dw}
D_r(x)=rx=(r^{\nu_{1}}x_1,\cdots, r^{\nu_{n}}x_n)\,,x \in \G\,,
\end{equation}
then the \textit{homogeneous dimension} $Q$ of the group is computed via the expression 
\[
Q:= \textnormal{Tr}A=\nu_{1}+\cdots+\nu_{n}\,.
\]
We invite the reader to refer to the examples of the Heisenberg group [Chapter 6 in \cite{FR16}] and of the Engel and Cartan groups \cite{Cha21} as the most well-studied cases of homogeneous groups.

\subsection{Representations and Rockland operators}
 	We denote by $\pi$ the representation of the group $\G$ on a separable Hilbert space $\mathcal{H}_\pi$, and keep the same notation for the representation of its Lie algebra $\mathfrak{g}$. The latter may as well be extended to the universal enveloping Lie algebra $\mathfrak{U}(\mathfrak{g})$, and is exactly the mapping $\pi: \mathfrak{U}(\mathfrak{g})\rightarrow \textnormal{End}(\mathcal{H}_{\pi}^{\infty})$, where $\mathcal{H}_{\pi}^{\infty} \subset \mathcal{H}_\pi$ is the space of smooth vectors. Elements of $\mathfrak{U}(\mathfrak{g})$ are precisely the left-invariant operators on $\G$, and as such, they can be written in a unique way as a finite sum of the form 
	\begin{equation}\label{pbw}
	\sum_{\alpha \in \mathbb{N}^c}c_{\alpha} X^\alpha\,,
\end{equation}
	where we use the notation $X^{\alpha}:=X_{1}^{\alpha_{1}}\cdots X_{n}^{\alpha_{n}}$, for some multi-index $\alpha=(\alpha_{1},\cdots,\alpha_{n}) \in \mathbb{N}^n$. 
The element $\pi \in \widehat{\G}$ acts on homogeneous of order $\nu$ left-invariant differential operator $\R$ in the following natural way:
	\[
\pi(\R)=\sum_{[\alpha]=\nu}c_\alpha\pi(X)^{\alpha}\,,
\]
where $[\alpha]=\nu_1 \alpha_1+\cdots+\nu_n \alpha_n$ stands for the homogeneous length of the multi-index $\alpha$, and we have used the notation $$\pi(X)^{\alpha}=\pi(X^\alpha)=\pi(X_{1}^{\alpha_1}\cdots X_{n}^{\alpha_n}),$$ 
where $X_j$ is of homogeneous of degree ${\nu}_j$ arising from the dilation weights in \eqref{dw}. 
	Important operators that are elements of $\mathfrak{U}(\mathfrak{g})$ are the so-called \textit{Rockland operators} on $\G$ as first appeared in \cite{Roc78}. The most trivial  example of a Rockland operator is the so-called \textit{sub-Laplacian} on a stratified group, which is exactly the sum of squares of vector fields that generate the first stratum of the Lie algebra $\mathfrak{g}$. In the trivial case $\G=\mathbb{R}^d$ the latter boils down to the Laplace differential operator. In their general consideration, Rockland operators are left-invariant differential operators on $\G$ that are densely defined on $\mathcal{D}(\G) \subset L^2(\G)$, that are homogeneous of some positive degree and satisfy the so-called \textit{Rockland condition}; that is when the operator $\pi(\mathcal{\R})$ is injective on the space of smooth vectors $\mathcal{H}_{\pi}^{\infty}\subset \mathcal{H}_{\pi}$ on which $\pi(\R)$ is densely defined. We mention that for the groups that we consider here we have $\mathcal{H}_\pi=L^2(\mathbb{R}^m)$ and  $\mathcal{H}_{\pi}^{\infty}=\mathcal{S}(\mathbb{R}^m)$, for some $m \in \mathbb{N}$; see [Corollary 4.1.2 in \cite{CG90}]. For a detailed discussion on the Rockland operator, and in particular for alternative characterisations of the Rockland condition, one can consult \cite{Bea77,HN79}. By the spectral theorem for unbounded self-adjoint operators, both $\R$ and $\pi(\R)$ have a spectral decomposition associated with the spectral measures of them; see e.g. Theorem VIII.6 in \cite{RS85}. The spectral properties of the operator $\pi(\R)$ were examined in \cite{TR97}. In \cite{HJL85} the authors proved that the spectrum of the operator $\pi(\R)$, for $\pi \in \widehat{\G}\setminus \{1\}$, is purely discrete and positive. This in turn allows for an infinite matrix representation of the following form:
\begin{equation}\label{repr.pr}
\pi(\R)=\begin{pmatrix}
\pi_{1}^{2} & 0 & \cdots & \cdots\\
0 & \pi_{2}^{2} & 0 & \cdots\\
\vdots & 0 & \ddots & \\
\vdots & \vdots & & \ddots
\end{pmatrix}\,.
\end{equation}
 Particularly, the spectrum of these operators in the setting of the Heisenberg and Engel group, has been studied using the representations of Taylor \cite{Tay84} and of Dixmier [p. 333 \cite{Dix57}], respectively, are investigated in e.g. [Chapter 6 in \cite{FR16}] and \cite{CDR18}. 
	\subsection{Group Fourier transform} For $f \in L^1(\G)$ and for $\pi \in \widehat{\G}$ the group Fourier transform of $f$ at $\pi$ is defined by 
	\[
 \mathcal{F}_{\G}f(\pi)\equiv \widehat{f}(\pi) \equiv \pi(f):= \int_{\G}f(x)\pi(x)^{*}\,dx\,,
 \]
 where we integrate against the binvariant Haar measure on $\G$. This produces a linear endomorphism on $\mathcal{H}_{\pi}$. Thanks to e.g. Kirillov's orbit method (e.g. \cite{CG90, Kir04}) one can explicitly construct the Plancherel measure $\mu$ on the dual $\widehat{\G}$. This brings the Fourier inversion formula at our disposal, and additionally allows for the extension of the Euclidean Plancherel formula to the setting of graded Lie groups; i.e., we have the isometry:
 	\begin{equation}\label{planc.id}
		\int_{\G} |f(x)|^2\,dx=\int_{\widehat{{\G}}}\|\pi(f)\|^{2}_{\textnormal{HS}}\,d\mu(\pi)\,,
		\end{equation}
 since the operator $\pi(f)$ is Hilbert-Schmidt; that is we have $\|\pi(f)\|^{2}_{\textnormal{HS}}:=\textnormal{Tr}(\pi(f)\pi(f)^{*})<\infty$. It makes sense to define the group Fourier transform $\mathcal{F}_{\G}(\R f)(\pi)$, and the functional calculus allows for the identity $\mathcal{F}_{\G}(\R f)(\pi)=\pi(\R)\pi(f)$. The last expression and the representation in \eqref{repr.pr} allow for a matrix representation of the operator $\mathcal{F}_{\G}(\R f)(\pi)$ of the form $\{\pi_{k}^2 \hat{f}(\pi)_{k,l} \}_{k,l \in \mathbb{N}}$ that shall be useful for our purposes.
 
\subsection{Sobolev spaces and embedding results}

The (homogeneous) \textit{$\R$-Sobolev spaces} in the setting of a graded Lie groups are defined by the second author and Fischer in \cite{FR17}: for a fixed positive homogeneous Rockland operator $\R$ of some homogeneous degree $\nu$, and for $s>0$, $p>1$, they are defined as the completion of $\mathcal{S}(\G)\cap \textnormal{Dom}(\R^{\frac{s}{\nu}})$ for the norm
\[
\|f\|_{\dot{L}^{p}_{s}(\G)}:=\|\R^{\frac{s}{\nu}}_{p}f\|_{L^p(\G)}\,,\quad f \in \mathcal{S}(\G)\cap \textnormal{Dom}(\R^{\frac{s}{\nu}}_{p})\,,
\]
where $\R_p$ is the maximal restriction of $\R$ to $L^{p}(\G)$.\footnote{When $p=2$, we will write $\R_2=\R$ for the self-adjoint extension of $\R$ on $L^2(\G)$.} 
Sobolev embeddings read as follows:  
for $1<\tilde{q}_0<q_0<\infty$ and for $a,b \in \mathbb{R}$ satisfying the condition
\[
b-a=Q \left( \frac{1}{\tilde{q}_0}-\frac{1}{q_0}\right)\,,
\]
we have the continuous inclusions $\dot{L}^{\tilde{q}_0}_{b}(\G) \subset \dot{L}^{q_0}_{a}(\G)$. The latter means that that, if
$f \in \dot{L}^{\tilde{q}_0}_{b}(\G)$, then, there exists a constant $C=C(\tilde{q}_0,q_0,a,b)>0$ independent of $f$ such that 
\begin{equation}
\label{inclusions}
\|f\|_{\dot{L}^{q_0}_{a}(\G)}\leq C \|f\|_{\dot{L}^{\tilde{q}_0}_{b}(\G)}\,.
\end{equation}

\par In the sequel we will make use of the following notation:
\begin{notation}
\begin{itemize}
    \item When we write $a \lesssim b$, we mean that there exists some constant $c>0$ (independent of any involved parameter) such that $a \leq c b$;
    \item if $\alpha=(\alpha_1,\cdots,\alpha_n) \in \mathbb{N}^n$ is some multi-index, then we denote by 
    \[
    |\alpha|=\sum_{i=1}^{n}\alpha_i\,,
    \]
    the length of it;
    \item for $s>0$ and for suitable $f \in \mathcal{S}^{'}(\G)$, we have introduced the following norm
    \begin{equation}\label{def.Hs-norm}
    \|f\|_{H^{s}(\G)}:=\|f\|_{\dot{L}^{2}_{s}(\G)}+\|f\|_{L^2(\G)}\,;
    \end{equation}
    \item when regulisations of functions/distributions on $\G$ are considered, they must be regarded as arising  via  convolution  with  Friedrichs-mollifiers;  that is, $\psi$ is a Friedrichs-mollifier, if it is a compactly supported smooth function with $\int_{\G}\psi\,dx=1$. Then the regularising net is defined as
\begin{equation}\label{mol}
	\psi_{\epsilon}(x)=\omega(\epsilon)^{-Q}\psi(D_{\omega(\epsilon)^{-1}}(x))\,,\quad \epsilon \in (0,1]\,,
\end{equation}
where $\omega(\epsilon)$ is a positive function converging to zero as $\epsilon \rightarrow 0$ and $Q$ stands for the homogeneous dimension of $\G$.
\end{itemize}
\end{notation}

\section{Estimates for the classical solution}

The results of the current section, and the forthcoming ones, fall into two categories; those where the potential $V$ as appears in \eqref{heat.eq} is assumed to be non-negative, and those where the potential can be of any sign. In the latter case the potential $V$ can only belong to the space $L^\infty(\G)$, while in the positive case it can also belong to the space  $L^{\frac{2Q}{\nu }}(\G)\cap L^{\frac{2Q}{\nu }}(\G)$, provided that the condition $Q > \nu$ is satisfied.
\begin{proposition}[Classical solution, case I: $V\geq 0$]\label{prop.pos.a}
Let $V \in L^{\infty}(\G)$, where $V \geq 0$, and suppose that $u_0 \in H^{\frac{ \nu}{2}}(\G)$. Then, there exists a unique solution $u \in C^{1}([0,T];L^2(\G))\cap C([0,T];H^{\frac{ \nu}{2}}(\G))$ to the Cauchy problem \eqref{heat.eq}, that satisfies the estimate 
	\begin{equation}\label{prop.pos.a.claim}
		\|u(t,\cdot)\|_{H^{\frac{ \nu}{2}}(\G)} \lesssim (1+\|V\|_{L^{\infty}(\G)}) \|u_0\|_{H^{\frac{ \nu}{2}}(\G)}\,,
		\end{equation}
		uniformly in $t \in [0,T]$.
\end{proposition}
\begin{proof}
Multiplying the equation \eqref{heat.eq} by $u_t$ and integrating over $\mathbb{G}$, we get
	\begin{equation*}\label{Re0,sch}
		\Re(\langle u_{t}(t,\cdot),u_t(t,\cdot)\rangle_{L^2(\mathbb{G})}+\langle \mathcal{R}u(t,\cdot),u_t(t,\cdot)\rangle_{L^2(\G)}+\langle V(\cdot)u(t,\cdot),u_t(t,\cdot) \rangle_{L^2(\G)})=0\,,
	\end{equation*}
for all $t \in [0,T]$, where $\Re$ stands for the real part of the above quantity. The above equation can be rewritten as
\[
\|u_t(t,\cdot)\|_{L^2(\G)}^{2}+\frac{1}{2}\partial_{t}\left[\| \sqrt{\R}u(t,\cdot)\|_{L^2(\G)}^{2}+  \| \sqrt{V}(\cdot)u(t,\cdot)\|_{L^2(\G)}^{2}\right]=0\,.
\]
We define the energy as 
\begin{equation*}\label{energy.est} E(t):=\|\sqrt{\R}u(t,\cdot)\|\L+\|\sqrt{V}(\cdot)u(t,\cdot)\|\L\,, 
\end{equation*}
which variation in time satisfies $\partial_{t}E(t)\leq 0$, and consequently  \begin{equation}\label{et=e0}E(t)\leq E(0)\,,\quad \text{for all}\quad t \in [0,T]\,.
\end{equation}
Straightforward computations give
\[
\|\sqrt{V}u_0\|\L \leq \|V\|_{L^{\infty}(\G)} \|u_0\|\L\,,
\]
so that equation \eqref{et=e0} can be rephrased as the uniform in time $t \in [0,T]$ estimates
\begin{equation}\label{pu.1.est}
\|\sqrt{V}u(t,\cdot)\|\L \lesssim \|\sqrt{\R}u_0\|\L+\|V\|_{L^{\infty}}\|u_0\|\L\,,
\end{equation}
and 
\begin{equation}\label{Ru.1.est}
\|\sqrt{\R}u(t,\cdot)\|\L \lesssim \|\sqrt{\R}u_0\|\L+\|V\|_{L^{\infty}}\|u_0\|\L\,.
\end{equation}
Recalling the definition of the sum of norms $H^s(\G)$ in \eqref{def.Hs-norm} we can estimate \eqref{pu.1.est} and \eqref{Ru.1.est} further by 
\begin{equation}\label{pu.2.est}
    \|\sqrt{V}u(t,\cdot)\|_{L^2(\G)} \lesssim \left( 1+\|V\|^{\frac{1}{2}}_{L^{\infty}(\G)}\right) \|u_0\|_{H^{\frac{ \nu}{2}}(\G)}\,,
\end{equation}
and
\begin{equation}\label{Ru.2.est}
    \|\sqrt{\R}u(t,\cdot)\|_{L^2(\G)}\lesssim \left( 1+\|V\|^{\frac{1}{2}}_{L^{\infty}(\G)}\right) \|u_0\|_{H^{\frac{ \nu}{2}}(\G)}\,,
\end{equation}
respectively.

To proceed we need to deal with the term $\|u(t,\cdot)\|_{L^2(\G)}$; let us rephrase the problem \eqref{heat.eq} as follows:
\begin{equation}\label{heat.eq2}
	\begin{cases}
		\partial_t u(t,x) +\mathcal{R}u(t,x)=f(t,x)\,,\quad (t,x)\in [0,T]\times \mathbb{G},\\
		u(0,x)=u_0(x),\,\\	
	\end{cases}       
\end{equation}
where we have set as a source term $f(t,x):=-V(x)u(t,x)$. Applying Duhamel's principle (see, e.g. \cite{Ev98}), we deduce that the solution to \eqref{heat.eq2} is given by 
\begin{equation}
	\label{sol.Duh}
	u(t,x)=(h_t \ast u_0)(x)+\int_{0}^{t
	}h_{t-s}\ast f_s(x)\,ds\,,
\end{equation}
where $f_s(\cdot)=f(s,\cdot)$, and $h_t(\cdot)=h(t,\cdot)$ with $h$ being the fundamental solution, or the heat kernel, associated to the Rockland operator $\R$. We recall (see, e.g. Theorem 4.2.7 \cite{FR16}) the following upper bound for the heat kernel norm
\[
\|h_t\|_{L^{1}(\G)}\leq 1\,.
\]
We take the $L^2(\G)$-norm of the expression on \eqref{sol.Duh}. Young's inequality implies for the convolution there
\begin{eqnarray}\label{lem1.last.est}
	\|u(t,\cdot)\|_{L^2(\G)} & \leq & \|h_t\|_{L^1(\G)}\|u_0\|_{L^2(\G)}+\int_{0}^{T} \|h_{t-s}\|_{L^1(\G)}\|f_s\|_{L^2(\G)}\,ds\nonumber\\
	& \leq & \|u_0\|_{L^2(\G)}+\int_{0}^{T}\|f_s\|_{L^2(\G)}\,ds\nonumber\\
	& \leq & \|u_0\|_{L^2(\G)}+\int_{0}^{T}\|V(\cdot)u(s,\cdot)\|_{L^2(\G)}\,ds\,.
\end{eqnarray}
Straightforward computations then give
\begin{equation}
    \label{prop1.l2-est}
    \|u(t,\cdot)\|_{L^2(\G)}\lesssim \|u_0\|_{H^{\frac{\nu}{2}}(\G)}+  \left(1+\|V\|_{L^{\infty}(\G)}\right)\|u_0\|_{H^{\frac{\nu}{2}}(\G)}\,.
\end{equation}
Putting together \eqref{Ru.2.est} and \eqref{prop1.l2-est} we obtain
\[
\|u(t,\cdot)\|_{L^2}^{2}+\|\sqrt{\R}u(t,\cdot)\|_{L^2}^{2} \lesssim (1+\|V\|_{L^\infty})\|u_0\|_{H^{\frac{\nu}{2}}}\,,
\]
as required, and the proof of Proposition \ref{prop.pos.a} is complete. 
\end{proof}

\begin{proposition}[Classical solution, case II: $V \geq 0$]\label{prop.pos.b}
Assume that $Q > \nu $, and let $V \in L^{\frac{2Q}{\nu }}(\G)\cap L^{\frac{Q}{\nu }}(\G)$, $V \geq 0$. If we suppose that $u_0 \in H^{\frac{\nu}{2}}(\G)$, then there exists a unique solution $u \in C^{1}([0,T];L^2(\G))\cap C([0,T];H^{\frac{ \nu}{2}}(\G))$ to the Cauchy problem \eqref{heat.eq} satisfying the estimate 
\begin{equation}
\label{prop.pos.b.claim}
\|u(t,\cdot)\|_{H^{\frac{ \nu}{2}}(\G)}\lesssim \|u_0\|_{H^{\frac{\nu}{2}}(\G)}  \left\{\left(1+\|V\|_{L^{\frac{2Q}{\nu }}(\G)}\right) \left(1+\|V\|_{L^{\frac{Q}{\nu }}(\G)}\right)^{\frac{1}{2}}\right\}\,,
\end{equation}
 uniformly in $t \in [0,T]$.
\end{proposition}
\begin{proof}
We apply H\"older's inequality to the norm $\|\sqrt{V} u_0\|\L$ for $1<q,q^{'}<\infty$, so that
\begin{equation}
\label{prop2.Hold}
 \|\sqrt{V} u_0\|\L \leq \|V\|_{L^{q^{'}}(\G)}\|u_{0}\|^{2}_{L^{2q}(\G)}\,,
\end{equation}
where $(q,q^{'})$ are conjugate exponents. We use the embeddings \eqref{inclusions} for $u_0 \in H^{\frac{\nu}{2}}(\G)$, $b=\frac{ \nu}{2}$, $a=0$, and $q_0=\frac{2Q}{Q-\nu }$. This leads to the estimate 
\begin{equation}
\label{prop2.hold1}
\|u_0\|_{L^{q_{0}}(\G)} \lesssim \|\sqrt{\R}u_0\|_{L^2(\G)}<\infty\,.
\end{equation}
Putting together \eqref{prop2.Hold} for $q=\frac{q_0}{2}$, so that $q'=\frac{Q}{\nu}$, and \eqref{prop2.hold1} we obtain
\begin{eqnarray}
\label{emb}
\|\sqrt{V}u_0\|\L & \lesssim & \|V\|_{L^{\frac{Q}{\nu }}(\G)}\|\sqrt{\R}u_0\|_{L^2(\G)}^{2}\nonumber\\
& \leq & \|V\|_{L^{\frac{Q}{\nu }}(\G)} \|u_0\|_{H^{\frac{\nu}{2}}(\G)}^{2}\,.
\end{eqnarray}
Using the energy bounds \eqref{et=e0} from 
 Proposition \ref{prop.pos.a}, we have 
 \[
\|\sqrt{\R}u(t,\cdot)\|\L+\|\sqrt{V}u(t,\cdot)\|\L\leq \left(1+\|V\|_{L^{\frac{Q}{\nu }}(\G)}\right)\|u_0\|_{H^{\frac{ \nu}{2}}(\G)}^{2}\,,
\]
and the latter, if rephrased, leads to the estimate
\begin{equation}\label{EQ:prop.pos.b.est1}
    \|u(t,\cdot)\|_{\dot{L}^{2}_{\frac{\nu}{2}}(\G)} \leq \left(1+\|V\|_{L^{\frac{Q}{\nu }}(\G)}\right)^{\frac{1}{2}}\|u_0\|_{H^{\frac{ \nu}{2}}(\G)}\,.
\end{equation}
The estimate \eqref{lem1.last.est} for the $L^2$-norm of $u(t,\cdot)$ combined with the estimate \eqref{emb} with $V$ and $u(t,\cdot)$ in place of $\sqrt{V}$ and $u_0$, respectively, and inequality \eqref{EQ:prop.pos.b.est1} gives

\begin{eqnarray}\label{EQ:prop.pos.b.est2}
    \|u(t,\cdot)\|_{L^2(\G)}& \lesssim & \|u_0\|_{L^2(\G)}+\|V^2\|^{\frac{1}{2}}_{L^{\frac{Q}{\nu}}(\G)}\|u(t,\cdot)\|_{H^{\frac{\nu}{2}}(\G)}^{2}\nonumber\\
    & \leq & \|u_0\|_{H^{\frac{\nu}{2}}(\G)}+\|V\|_{L^{\frac{2Q}{\nu}}(\G)}\left(1+\|V\|_{L^{\frac{Q}{\nu }}(\G)}\right)^{\frac{1}{2}}\|u_0\|_{H^{\frac{ \nu}{2}}(\G)}\,.
\end{eqnarray}
 Putting together the estimate \eqref{EQ:prop.pos.b.est2} with the one in \eqref{EQ:prop.pos.b.est1} we get the requested estimate for the $H^{\frac{\nu}{2}}$-norm of $u(t,\cdot)$. The proof of Proposition \ref{prop.pos.b} is complete.
\end{proof}
The following result applies to the case where the potential $V$ is of any sign. In this general case, we can ease the assumption on the initial data $u_0$, and, as a result, we control only the $L^2$-norm of the classical solution. In particular we have: 
\begin{proposition}[Classical solution, case $V \in \mathbb{R}$]\label{prop.neg.a}
Let $u_0 \in L^2(\G)$ and suppose that $V\in L^\infty(\G)$ is a real function/distribution. Then, there exists a unique solution $u \in C([0,T];L^2(\G))$ to the Cauchy problem \eqref{heat.eq} that satisfies the estimate 
\begin{equation}
    \label{prop.neg.a.claim}
    \|u(t,\cdot)\|_{L^2(\G)} \lesssim \exp( t\|V\|_{L^\infty(\G)})\|u_0\|_{L^2(\G)}\,,
\end{equation}
for all $t\in [0,T]$.
\end{proposition}

	\begin{proof}
		We multiply the heat  equation \eqref{heat.eq} by $u$ and integrating on the group $\G$ and obtain the equality:
		\begin{equation*}\label{Re=0,2}
			\Re(\langle u_t(t,\cdot),u(t,\cdot)\rangle_{L^2(\mathbb{G})}+\langle \mathcal{R}u(t,\cdot),u(t,\cdot)\rangle_{L^2(\G)}+\langle V(\cdot)u(t,\cdot),u(t,\cdot) \rangle_{L^2(\G)})=0\,,
		\end{equation*}
		for all $t \in [0,T]$. The last equality can be rewritten as
		\begin{equation}\label{prop2.any.sob}
		    \frac12	\partial_t\|u(t,\cdot)\|_{L^2(\G)}^2+\|\sqrt{\R}u(t,\cdot)\|_{L^2(\G)}^2+\|\sqrt{V^{+}}u(t,\cdot)\|_{L^2(\G)}^2= \|\sqrt{V^{-}}u(t,\cdot)\|_{L^{2}(\G)}^2
		\end{equation}
		for all $t \in [0,T]$, where $V^+,V^-$ are the positive and negative parts of $V$, respectively.
		Straightforward computations yield 
		\[\|\sqrt{V^{-}}u(t,\cdot)\|_{L^{2}(\G)}\leq \|\sqrt{V^-}\|_{L^{\infty}(\G)}\|u(t,\cdot)\|_{L^2(\G)} \leq \|\sqrt{|V|}\|_{L^{\infty}(\G)}\|u(t,\cdot)\|_{L^2(\G)}\,. \]
		Thus, equation \eqref{prop2.any.sob} implies for the time-derivative of $\|u(t,\cdot)\|_{L^2(\G)}$:
		\begin{equation}\label{before.G}	\frac12 \partial_t\|u(t,\cdot)\|_{L^2(\G)}^2 \leq \|\sqrt{|V|}\|_{L^{\infty}(\G)}^2 \|u(t,\cdot)\|_{L^2(\G)}^2\,,
		\end{equation}
		for all $t \in [0,T]$.
	Hence we can apply Gr\"{o}nwall's inequality to \eqref{before.G}, so that 
	\[
	\|u(t,\cdot)\|_{L^2(\G)} \leq \|u_0\|_{L^2(\G)}\exp \left(\int_{0}^{t}\|V\|_{L^\infty(\G)}\,ds \right)\,,
	\]
	and we obtain the desired estimate \eqref{prop.neg.a.claim}. This completes the proof of Proposition \ref{prop.neg.a}.
	\end{proof}
\begin{remark}\label{rem.any.sob}
	Let us point out that, under the assumptions of Proposition \ref{prop.neg.a}, the solution $u\in L^2(\G)$ to the Cauchy problem \eqref{heat.eq}, belongs also to the Sobolev space of order $\frac{\nu}{2}$: Indeed, by \eqref{prop2.any.sob} we get 
	\[
\|u(t,\cdot)\|_{\dot{L}^{2}_{\frac{\nu}{2}}(\G)}=	\|\sqrt{\R}u(t,\cdot)\|_{L^2} \leq \|\sqrt{V^-}\|_{L^\infty}^2 \|u\|_{L^2}^{2}\,,
	\]
	with the right-hand side of the last inequality being, by the assumptions of Proposition \ref{prop.neg.a}, finite.
	\end{remark}
	

\section{Very weak well-posedeness}

In what follows we shall show that the notion of the very weak solution as introduced in \cite{GR14} is applicable to the heat equation \eqref{heat.eq} we consider here. Once the existence of the very weak solution is proven, we then investigate how the latter depends on the $\omega(\epsilon)$-scale or the approximation of the associated, to the initial data $u_0$ and to the coefficient $V$, nets. In other words, we analyse the stability, or formally the \textit{uniqueness}, of the very weak solution under the aforesaid modifications. 

Let us recall the definitions of moderateness for a net of functions/distributions in space (and in time) as in \cite{CTR21} and \cite{CTR21b}.

\begin{definition}[Moderateness]
Let $X$ be a normed space of functions/distributions and let $\omega(\epsilon)$ be as in \eqref{mol}.
	\begin{enumerate}
		\item   A net of functions/distributions $(f_\epsilon)_\epsilon \in X$ is said to be \textit{$X$-moderate} if there exists $N \in \mathbb{N}$ such that \[
		\|f_\epsilon\|_{X}\lesssim \omega(\epsilon)^{-N}\,,
		\]
		for all $\epsilon \in (0,1]$.
		\item A net of functions/distribitions $(u_\epsilon)_\epsilon=(u_\epsilon(t,\cdot))_\epsilon$ in $C([0,T]; X)$ is said to be \textit{$C([0,T]; X)$-moderate} if there exists $N \in \mathbb{N}$ such that 
		\[
		\sup_{t \in [0,T]}\|u_\epsilon(t,\cdot)\|_{X}\lesssim \omega(\epsilon)^{-N}\,,
		\]
		for all $\epsilon \in (0,1]$.
	\end{enumerate}
\end{definition}
\begin{remark}\label{rem:e'}
Trivially, for nets that arise as regularisations of a distribitional function in some normed space $X$ via a mollifier the $X$-moderate assumption is satisfied. More generally, see Proposition \ref{prop.e'} below, the global structure of compaclty supported distributions $\mathcal{E}'(\G)$ on the group, gives rise to $L^p(\G)$-moderate regularised such nets. We note that following the arguments developed in the proof of Proposition \ref{prop.e'} we can show that the regularised net $(v_\epsilon)_\epsilon$ is $\dot{L}^{2}_{s}(\G)$, and so also $H^{s}(\G)$-moderate for every $s>0$.
\end{remark}
 Let us recall the following result as in 
  \cite{CTR21,CTR21b}:

\begin{proposition}
\label{prop.e'}
Let $v \in \mathcal{E}^{'}(\G)$, and let 
$v_\epsilon=v*\psi_\epsilon$ be obtained as
the convolution of $v$ with a mollifier $\psi_\epsilon$ as in \eqref{mol}.
Then the regularising net $(v_\epsilon)_\epsilon$ is $L^{p}(\G)$-moderate for any $p \in [1,\infty]$.
\end{proposition}

Describing the uniqueness of the very weak solutions amounts to ``measuring'' the changes on involved associated nets: \textit{negligibility} conditions for nets of functions/di-stributions read as follows:
\begin{definition}[Negligibility]
\label{defn.negl}
{Let $Y$ be a normed space of functions on $\G$.
Let $(f_\epsilon)_\epsilon $, $(\tilde{f}_\epsilon)_\epsilon$ be two nets. Then, the net $(f_\epsilon-\tilde{f}_\epsilon)_\epsilon$ is called $Y$-{\em negligible}, if the following condition is satisfied
    \begin{equation}\label{def.cond.negl}
    \|f_\epsilon-\tilde{f}_\epsilon\|_{Y}\lesssim \omega(\epsilon)^k\,,
    \end{equation}
    for all $k \in \mathbb{N}$, $\epsilon \in (0,1]$. In the case where $f_\epsilon=f_\epsilon(t,x)$ is a net depending also on $t \in [0,T]$, then the {\em negligibility condition} \eqref{def.cond.negl} can be regarded as
    \[
    \|f_\epsilon(t,\cdot)-\tilde{f}_\epsilon(t,\cdot)\|_{Y}\lesssim \omega(\epsilon)^k\,,\quad \forall k \in \mathbb{N}\,,
    \]
    uniformly in $t \in [0,T]$.} The constant in the inequality \eqref{def.cond.negl} can depend on $k$ but not on $\omega$.
\end{definition}

Definition \ref{def.vws.pos} and Definition \ref{def.vws.any} below introduce the notion of the very weak solution to the Cauchy problem \eqref{heat.eq}. The two definitions differ depending on the sign of the potential $V$ and shall be used accordingly. However, in principle, one could simply suggest to use only Definition  \ref{def.vws.any} where $V \in \mathbb{R}$, but in the case of a non-negative $V$ as in Definition \ref{def.vws.pos} the moderateness assumptions are more relaxed.

Let us clarify that in both definitions the approximating net $u_{0,\epsilon}$ shall be regarded as the regularisation of $u_0$ via convolution with a Fredrichs-mollifier $\psi$ as in \eqref{mol}. Regarding the potential $V$, the approximating net $V_\epsilon$ is an expansion of such idea, that encloses the cases where $V$ is singular. In that manner it is possible to consider $V=\delta$ or $V=\delta^2$. In the latter case we realise $V$ as an approximating family or in the Colombeau sense, and define $V_\epsilon$ as $V_\epsilon:=\psi_{\epsilon}^{2}$. We refer to Remark \ref{rem.negl} for a complementary discussion. 

Before stating the below definitions let us formulate the Cauchy problem \eqref{heat.eq} in its ``$\epsilon$-parametrised'' version:

\begin{equation}
    \label{reg.heat.eq}
    \begin{cases}
\partial_t u_\epsilon(t,x) +\mathcal{R}u_\epsilon(t,x)+V_\epsilon(x)u_\epsilon(t,x)=0\,,\quad (t,x)\in [0,T]\times \mathbb{G}\,,\\
u_\epsilon(0,x)=u_{0,\epsilon}(x), \, x \in \G\,,	
\end{cases}  
\end{equation}

\begin{definition}[Very weak solution, case $V\geq 0$]\label{def.vws.pos}
Let $V$ in \eqref{heat.eq} be non-negative. If there exists
\begin{itemize}
    \item a $L^\infty$-moderate, or (provided that $Q> \nu$) a $L^{\frac{Q}{\nu}}(\G)\cap L^{\frac{2Q}{\nu}}(\G)$-moderate approximation of $V$; 
    \item a $H^{\frac{\nu}{2}}(\G)$-moderate approximation of $u_0$;
    \item and the solution to the Cauchy problem \eqref{reg.heat.eq} $(u_\epsilon)_\epsilon$ is $C([0,T];H^{\frac{\nu}{2}}(\G))$-moderate,
\end{itemize}
then, $(u_\epsilon)_\epsilon$ is said to be \textit{a very weak solution} to the Cauchy problem \eqref{heat.eq} in the case of a non-negative potential $V$. 
\end{definition}

\begin{definition}[Very weak solution, case $V \in \mathbb{R}$]\label{def.vws.any}
Let $V$ be any real function. If there exists
\begin{itemize}
    \item a $L^\infty$-moderate approximation of $V$;
    \item a $H^{\frac{\nu}{2}}(\G)$-moderate approximation of $u_0$;
    \item and the solution to the Cauchy problem \eqref{reg.heat.eq} $(u_\epsilon)_\epsilon$ is $C([0,T];L^2(\G))$-moderate,
\end{itemize}
then, $(u_\epsilon)_\epsilon$ is said to be \textit{a very weak solution} to the Cauchy problem \eqref{heat.eq} in the case of a real potential $V$.
 \end{definition}

As mention earlier in Remark \ref{rem:e'}, proving the existence of a very weak solution means to prove that there exist suitable approximations of the initial data $u_0$ and of the coefficient $V$ such that the approximated problem \eqref{reg.heat.eq} has a suitable moderate solution. In the next two theorems we prove such existence under suitable assumptions depending on the sign of $V$.

\begin{theorem}
[Existence of the very weak solution, case $V\geq 0$]
\label{th.ex.heat.pos}
Let $V$ be non-negative. Let also $u_0 \in H^{\frac{\nu}{2}}(\G)\cup \mathcal{E}'(\G)$ and $V \in L^{\infty}(\G)\cup \mathcal{E}'(\G)$, or (provided that $Q > \nu$), $V \in L^{\frac{Q}{\nu}}(\G)\cap L^{\frac{2Q}{\nu}}(\G)$.
Then, the very weak solution to the Cauchy problem \eqref{heat.eq} exists.
\end{theorem}

\begin{proof}
Let $u_0 \in H^{\frac{\nu}{2}}(\G)\cap \mathcal{E}'(\G)$. We first treat the case where $V\in L^{\infty}(\G)\cup \mathcal{E}'(\G)$. If $V_\epsilon, u_{o,\epsilon}$ are regularised (via convolution with a Friedrichs' mollifier) nets, then they satisfy the moderateness assumptions of Definition \ref{def.vws.pos}; i.e., we have 
\[
\|V_\epsilon\|_{L^\infty(\G)},\|u_{0,\epsilon}\|_{H^{\frac{\nu}{2}}(\G)} \lesssim \epsilon^{-n_0}\,,
\]
where we have chosen $\omega(\epsilon)=\epsilon$, for some $n_0\in \mathbb{N}$. Therefore, using the estimate \eqref{prop.pos.a.claim} we have 
\begin{eqnarray*}
\|u_\epsilon(t,\cdot)\|_{H^{\frac{\nu}{2}}(\G)}& \lesssim & (1+\|V_\epsilon\|_{L^\infty(\G)})\|u_{0,\epsilon}\|_{H^{\frac{\nu}{2}}(\G)}\\
& \lesssim & \epsilon^{-n_0} \times \epsilon^{-n_0}\\
& = & \epsilon^{-2n_0}\,,
\end{eqnarray*}
uniformly in $t$, and this shows that the net $u_{0,\epsilon}$ is $C([0,T];H^{\frac{\nu}{2}}(\G))$-moderate, as required. Next we treat the case where $V \in L^{\frac{Q}{\nu}}(\G) \cap L^{\frac{Q}{2\nu}}(\G)$. Consequently, for the net $V_\epsilon$ as before, the $L^{\frac{Q}{\nu}}(\G) \cap L^{\frac{Q}{2\nu}}(\G)$-moderateness assumptions for $\omega(\epsilon)=\epsilon$ and for some $n_1 \in \mathbb{N}$ are guaranteed. Thus, using the estimate \eqref{prop.pos.b.claim}, we can write
\begin{eqnarray*}
\|u(t,\cdot)\|_{H^{\frac{\nu}{2}}(\G)} & \lesssim & \|u_{0,\epsilon}\|_{H^{\frac{\nu}{2}}(\G)}  \left\{\left(1+\|V_\epsilon\|_{L^{\frac{2Q}{\nu }}(\G)}\right) \left(1+\|V_\epsilon\|_{L^{\frac{Q}{\nu }}(\G)}\right)^{\frac{1}{2}}\right\}\\
& \lesssim & \epsilon^{-n_0} \times \epsilon^{-n_1} \times \epsilon^{-\frac{n_1}{2}}\\
& = & \epsilon^{-n}\,,
\end{eqnarray*}
for some $n$, and the proof of Theorem \ref{th.ex.heat.pos} is complete.
\end{proof}

\begin{theorem}
[Existence of the very weak solution, case $V \in \mathbb{R}$]
\label{th.ex.heat.any}
Let $V$ be a real function. Let also $u_0 \in L^2(\G) \cup \mathcal{E}'(\G)$ and $V \in L^{\infty}(\G)\cup \mathcal{E}'(\G)$. Then, the very weak solution to the Cauchy problem \eqref{heat.eq} exists.
\end{theorem}
\begin{proof}
By the assumptions of the theorem, the $L^{\infty}(\G)$-moderateness assumptions of the net $V_\epsilon$ are satisfied; i.e., for some $n_0 \in \mathbb{N}$ we have 
\[
\|V_\epsilon\|_{L^\infty (\G)} \lesssim \omega(\epsilon)^{-n_0}\,.
\]
Taking $\omega$ as $\omega(\epsilon):=(\log \epsilon^{-n_0})^{-\frac{1}{n_0}}$, where $n_0 \in \mathbb{N}$ is such that the moderateness assumptions can be expressed in the form
\[
\|V_\epsilon\|_{L^\infty (\G)} \lesssim \omega(\epsilon)^{-n_0}= \log (\epsilon^{-n_0})\,.
\]
For the net $u_{0,\epsilon}$ the $L^2$-moderateness can be simply regarded as
\[
\|u_{0,\epsilon}\|_{L^2(\G)}\lesssim \epsilon^{-n_1}\,, \quad n_1 \in \mathbb{N}\,.
\]
Using the estimate \eqref{prop.neg.a.claim} for the solution to \eqref{heat.eq} in the case of a real potential we deduce that  
\begin{eqnarray*}
  \|u_\epsilon(t,\cdot)\|_{L^2(\G)}  &\lesssim & \exp(t\|V_\epsilon\|_{L^\infty(\G)})\|u_{0,\epsilon}\|_{L^2(\G)}\\
  & \lesssim & \epsilon^{-tn_0} \times \epsilon^{-n_1}\\
  & \leq & \epsilon^{-Tn_0-n_1}\\
  & = & \epsilon^{-n}\,,
\end{eqnarray*}
for some $n \in \mathbb{N}$, and this completes the proof of Theorem \ref{th.ex.heat.any}.
\end{proof}

\begin{corollary}\label{cor.any.sob}
Under the assumptions of Theorem \ref{th.ex.heat.any} the very weak solution $u$ to the Cauchy problem \eqref{heat.eq} is also $H^{\frac{\nu}{2}}(\G)$-moderate. 
\end{corollary}
\begin{proof}
The proof is an immediate consequence of Theorem \ref{th.ex.heat.any} and takes into account Remark \ref{rem.any.sob}.
\end{proof}
\begin{remark}
It is evident that, in the case where $V \geq 0$ and $V_\epsilon$ is a $L^{\infty}(\G)$-moderate net, one can ease the assumptions on the moderateness of $u_{0,\epsilon}$ as in Definition \ref{def.vws.any} where $V \in \mathbb{R}$ provided the function $\omega(\epsilon)$ as in the moderateness of $V_\epsilon$ is sharp enough; i.e., we require 
\[
\omega(\epsilon)^{-n_1}\leq \log \epsilon^{-n_0}\,,
\]
for some $n_0,n_1 \in \mathbb{N}$. In this case the $C([0,T];L^2(\G))$-moderateness of $u_\epsilon$ shall be proved using the estimate \eqref{prop.neg.a.claim} that applies to the case of a real potential $V$.
\end{remark}


As mentioned in the beginning of this section, proving the well-posedness to the Cauchy problem \eqref{heat.eq} in the very weak sense, amounts to proving that a very weak solution exists and it is unique modulo negligible nets. Here negligibility has to be understood under appropriate choices of norms.

Before giving the formal definition of the uniqueness of the very weak solution in our setting, see Definiton \ref{defn.uniq.heat.pos}, case $V \geq 0$, and Definition \ref{defn/uniq.any}, case $V \in \mathbb{R}$, below, let us state the ``$\epsilon$-paremetrised problems to be considered:  
 \begin{equation}\label{heat.reg.tild}
		\begin{cases}
			\partial_{t}u_\epsilon(t,x) +\mathcal{R} u_\epsilon (t,x)+V_\epsilon(x)u_\epsilon(t,x)=0\,,\quad (t,x)\in [0,T]\times \mathbb{G},\\
			u_\epsilon(0,x)=u_{0,\epsilon}(x),\; x \in \G\,,	
		\end{cases}       
	\end{equation}
	and 
	\begin{equation}\label{heat.ret.notil}
		\begin{cases}
			\partial_{t}\tilde{u}_\epsilon(t,x) +\mathcal{R}\tilde{u}_\epsilon (t,x)+\tilde{V}_\epsilon(x)\tilde{u}_\epsilon(t,x)=0\,,\quad (t,x)\in [0,T]\times \mathbb{G},\\
			\tilde{u}_\epsilon(0,x)=\tilde{u}_{0,\epsilon}(x),\; x \in \G\,.	
		\end{cases}       
	\end{equation}
\begin{definition}[Uniqueness of the very weak solution, case $V\geq 0$]\label{defn.uniq.heat.pos}
Let $V$ be non-negative. Let also $X$ and $Y$ be normed spaces of functions/distributions on $\G$.
We say that the Cauchy problem \eqref{heat.eq} has an \textit{$(X,Y)$-unique very weak solution}, if 
\begin{itemize}
    \item for all $X$-moderate nets
	$V_\epsilon, \tilde{V}_\epsilon,$ such that $(V_\epsilon-\tilde{V}_\epsilon)_\epsilon$ is $Y$-negligible; and
	\item for all $H^{\frac{ \nu}{2}}(\G)$-moderate regularisations $u_{0,\epsilon}$, $\tilde{u}_{0,\epsilon}$ such that $(u_{0,\epsilon}-\tilde{u}_{0,\epsilon})_\epsilon$ is $L^2(\G)$-negligible,
\end{itemize}
 the net $(u_\epsilon-\tilde{u}_\epsilon)$ is $L^2(\G)$-negligible, where $(u_\epsilon)_\epsilon$ and $(\tilde{u}_\epsilon)_\epsilon$ are the families of solutions corresponding to the $\epsilon$-parametrised problems \eqref{heat.reg.tild} and \eqref{heat.ret.notil}, respectively.
\end{definition}

\begin{definition}\label{defn/uniq.any}
[Uniqueness of the very weak solution, case $V \in \mathbb{R}$]
In the case of a real potential $V$ the definition is similar to Definition \ref{defn.uniq.heat.pos} but the moderateness assumption regarding the nets $u_{o,\epsilon}, \tilde{u}_{0,\epsilon}$ can be relaxed to being $L^2(\G)$-moderate.
\end{definition}

\begin{remark}\label{rem.negl}
In the work \cite{ARST21c}, Definitions 4,5 and 8 on the uniqueness of the very weak solution to the heat equation in the Euclidean setting considered there, are a less meticulous version of Definitions \ref{defn.uniq.heat.pos} and \ref{defn/uniq.any} that we give here. In particular, the previous and the current definitions differ with respect to the assumptions on the asymptotic behaviour of the nets $V_\epsilon$ and $\tilde{V}_\epsilon$: here we enlarge the requirement that the aforementioned nets approximate $V$ as we simply assume that the difference of nets $V_\epsilon-\tilde{V}_\epsilon$ is $L^\infty$-negligible. In this manner the initial idea to define uniqueness in terms of the stability  of the very weak solution under negligible changes on the coefficient $V$ remains, but we drop the previous requirement for the nets $V_\epsilon$ and $\tilde{V}_\epsilon$ to be regularisations of $V$. 
\end{remark}
Let us highlight Remark \ref{rem.negl} by the the next example where we provide cases of nets $V_\epsilon$ and $\tilde{V}_\epsilon$ that fall into the assumptions of the definitions of the uniqueness here, and are not covered but those in \cite{ARST21c}: 
\begin{example}
\begin{itemize}
    \item Let $V$ be any function/distribution and the nets $V_\epsilon$ and $\tilde{V}_\epsilon$ to be defined as:
    \begin{equation}\label{mex}
        \begin{cases}
        V_\epsilon:=V \ast \psi_\epsilon\,, \text{where}\,  \psi_\epsilon:=\omega(\epsilon)^{-Q} \psi \circ D_{\omega(\epsilon)^{-1}}\nonumber\,,\\
        \tilde{V}_\epsilon=V_\epsilon+e^{-1/\epsilon}\,.
        \end{cases}
    \end{equation}
    \item For $V=\delta^2$ and nets $V_\epsilon$ and $\tilde{V}_\epsilon$ to be given by: 
    \begin{equation*}
        \begin{cases}
        V_\epsilon=\psi_{\epsilon}^{2}\\
         \tilde{V}_\epsilon=V_\epsilon+e^{-1/\epsilon}\,.
        \end{cases}
    \end{equation*}
\end{itemize}
The above considerations give rise to nets $V_\epsilon-\tilde{V}_\epsilon$ that are $L^\infty$-negligible, and therefore, satisfy the assumptions described in our definitions of uniqueness of the very weak solution to the Cauchy problem \eqref{heat.eq}. 
\end{example}

In the remaining section we show the well-posedeness to the Cauchy problem \eqref{heat.eq} in any feasible $(X,Y)$-very weak sense. 

The following technical lemma is useful for our purposes in the case of a non-negative $V$.
\begin{lemma}
\label{techn.lem}
Let $V$ be non-negative. Assume also that $u_0 \in L^2(\G)$ and $V \in L^{\infty}(\G)$, or $V \in L^{\frac{Q}{\nu}}(\G)\cap L^{\frac{2Q}{\nu}}(\G)$, provided that $Q> \nu$. Then, for the unique solution $u$ to the Cauchy problem \eqref{heat.eq} we have the energy estimate
\begin{equation}\label{en.est}
\|u(t,\cdot)\|_{L^2(\G)}\leq \|u_0\|_{L^2(\G)}\,,
\end{equation}
for all $t \in [0,T]$.
\end{lemma}
\begin{proof}
If we multiply equation \eqref{heat.eq} by $u$, and integrate over $\G$ we derive
\[
	\Re(\langle u_t(t,\cdot),u(t,\cdot)\rangle_{L^2(\mathbb{G})}+\langle \mathcal{R}u(t,\cdot),u(t,\cdot)\rangle_{L^2(\G)}+\langle V(\cdot)u(t,\cdot),u(t,\cdot) \rangle_{L^2(\G)})=0\,,
\]
where the last inequality can be rewritten as 
\[
\frac12 \partial_t \|u(t,\cdot)\|\L = -\|\sqrt{\R}u(t,\cdot)\|\L-\|\sqrt{V}(\cdot)u(t,\cdot)\|\L \leq 0\,,
\]
and the claim \eqref{en.est} follows.
\end{proof}

\begin{theorem}[Uniqueness of the very weak solution, case I: $V \geq 0$]
\label{thm.un.pos.1}
Let $V$ be non-negative. Let also $u_0 \in H^{\frac{\nu}{2}}(\G)\cup \mathcal{E}'(\G)$. The following statements hold true:
\begin{itemize}
    \item If $V \in L^{\infty}(\G)\cup \mathcal{E}'(\G)$, then the very weak solution to the Cauchy problem \eqref{heat.eq} is $(L^{\infty}(\G),L^\infty (\G))$-unique;
    \item if $Q > \nu$ and $V \in L^{\infty}(\G)\cup \mathcal{E}'(\G)$, then the very weak solution to the Cauchy problem \eqref{heat.eq} is $(L^{\frac{Q}{\nu}}(\G)\cap L^{\frac{2Q}{ \nu}}(\G), L^\infty (\G))$-unique.
\end{itemize}
\end{theorem}
\begin{proof}
We denote by $u_\epsilon$ and $\tilde{u}_\epsilon$ the families of solutions to the Cauchy problems \eqref{heat.ret.notil} and \eqref{heat.reg.tild}, respectively. Setting $U_\epsilon$ to be the difference of these nets $U_\epsilon:=u_\epsilon(t,\cdot)-\tilde{u}_\epsilon(t,\cdot)$, then $U_\epsilon$ solves
\begin{eqnarray}
\label{forU1}
\begin{cases}
\partial_t U_\epsilon(t,x)+\R U_\epsilon(t,x)+V_\epsilon(x)U_\epsilon(t,x)=f_\epsilon(t,x)\,,\quad (t,x)\in [0,T] \times \G\,,\\
U_\epsilon(0,x)=(u_{0,\epsilon}-\tilde{u}_{0,\epsilon})(x)\,,\quad x \in \G\,,
\end{cases}
\end{eqnarray}
where we set $f_\epsilon(t,x):=(\tilde{V}_\epsilon(x)-V_\epsilon(x))\tilde{u}_\epsilon(t,x)$ for the mass term to the inhomogeneous Cauchy problem \eqref{forU1}.

We aim to express the solution to \eqref{forU1} in terms of the solutions to the corresponding homogeneous problems: fix $\sigma \in [0,T]$, and let $W_\epsilon(t,x)$ and $\tilde{U}_\epsilon(t,x;\sigma)$ being the solutions to following the Cauchy problems
\begin{equation*}
\begin{cases*}
	\partial_{t}W_\epsilon(t, x)+\R W_\epsilon(t, x)+V(x)_\epsilon(x) W_\epsilon(t, x)=0\,,\quad\text{in}\, [0,T] \times \G,\\
	W_\epsilon(t, x)=(u_{0,\epsilon}-\tilde{u}_{0,\epsilon})(x)\,\quad \text{on}\,\{t=0\}\times \G\,,
\end{cases*}
\end{equation*}
and
\begin{equation*}\label{hom.Ve}
\begin{cases*}
	\partial_{t}\tilde{U}_\epsilon(t, x;\sigma)+\R \tilde{U}_\epsilon(t, x;\sigma)+V_\epsilon(x) \tilde{U}_\epsilon(t, x;\sigma)=0\,,\quad\text{in}\, (\sigma,T] \times \G,\\
	\tilde{U}_\epsilon(t, x;\sigma)=f_\epsilon(\sigma,x)\,\quad \text{on}\,\{t=\sigma\}\times \G\,.
\end{cases*}
\end{equation*}
Thus, with the above considerations, we can write as an application of Duhamel's principle:
\begin{equation*}\label{duh.prin}
U_\epsilon(t,x)=W_\epsilon(t,x)+\int_{0}^{t}\tilde{U}_\epsilon(t-\sigma,x;\sigma)\,d\sigma\,.
\end{equation*}
Passing to the $L^2$-norm of the latter expression we get
\[
	\|U_\epsilon(t,\cdot)\|_{L^2(\G)} \leq \|W_\epsilon(t,\cdot)\|_{L^2(\G)}+ \left\| \int_{0}^{t}\tilde{U}_\epsilon(t-\sigma,\cdot;\sigma)\,d\sigma \right \|_{L^2(\G)}\,.
	\]
	An application of Minkowski's integral inequality, combined with the energy upper bound \eqref{en.est} for both solutions $W_\epsilon$ and $\tilde{U}_\epsilon$ to the Cauchy problems as above, allow to estimate $\|U_\epsilon(t,\cdot)\|_{L^2}$  further as given below:
	\begin{eqnarray}\label{mink}
\|U_\epsilon(t,\cdot)\|_{L^2(\G)} & \leq & \|W_\epsilon(t,\cdot)\|_{L^2(\G)}+\int_{0}^{T}\|\tilde{U}_\epsilon(t-\sigma,\cdot;\sigma)\|_{L^2(\G)}\,d\sigma\nonumber\\
& \leq & \|u_{0,\epsilon}-\tilde{u}_{0,\epsilon}\|_{L^2(\G)}+\int_{0}^{T}\|f_\epsilon(\sigma,\cdot)\|_{L^2(\G)}\,d\sigma\,.
\end{eqnarray}
Since
\[
\|f_\epsilon(\sigma,\cdot)\|_{L^2(\G)}\leq \|\tilde{V}_\epsilon-V_\epsilon\|_{L^\infty(\G)}\|\tilde{u}_\epsilon(\sigma,\cdot)\|_{L^2(\G)}\,,
\]
using \eqref{mink} we obtain 
\[
\|\tilde{u}_\epsilon(\sigma,\cdot)\|_{L^2(\G)} \leq \|u_{0,\epsilon}-\tilde{u}_{0,\epsilon}\|_{L^2(\G)}\int_{0}^{T} \|\tilde{u}_\epsilon(\sigma,\cdot)\|_{L^2(\G)}\,d\sigma\,.
\]
Finally, taking into account the negligibility of the nets $u_{0,\epsilon}-\tilde{u}_{0,\epsilon}$ and $\tilde{V}_\epsilon-V_\epsilon$ we get
\[
\|U_\epsilon(t,\cdot)\|_{L^2(\G)} \lesssim \omega_1(\epsilon)^{N}+\omega_2(\epsilon)^{\tilde{N}}  \int_{0}^{T}\omega_3(\epsilon)^{-N_1}\,d\sigma  \,,
\]
for some $N_1 \in \mathbb{N}$, and for all $N,\tilde{N} \in \mathbb{N}$, since $\tilde{u}_\epsilon$ is $H^{\frac{\nu}{2}}(\G)$-moderate, and so also $L^2(\G)$-moderate. The last estimate holds true uniformly in $t$, and this completes the proof of Theorem \ref{thm.un.pos.1}.
\end{proof}

Still in the case where $V \geq 0$, alternative to Theorem \ref{thm.un.pos.1} conditions on the negligibility of the net $V_\epsilon-\tilde{V}_\epsilon$ that guarantee the very weak well-posedness of \eqref{heat.eq} are given in the next theorem.
\begin{theorem}
\label{thm.un.pos2}
Let $V$ be non-negative. Let also $u_0 \in H^{\frac{\nu}{2}}(\G)\cup \mathcal{E}'(\G)$. The following statements hold true:
\begin{itemize}
    \item If $V \in L^{\infty}(\G)\cup \mathcal{E}'(\G)$, then the very weak solution to the Cauchy problem \eqref{heat.eq} is $(L^\infty (\G),L^{\frac{2Q}{\nu}}(\G))$-unique;
    \item if $Q > \nu$ and $V \in L^{\frac{Q}{\nu}}(\G)\cap L^{\frac{Q}{2\nu}}(\G)$, then the very weak solution to the Cauchy problem \eqref{heat.eq} is $(L^{\frac{Q}{\nu}}(\G)\cap L^{\frac{2Q}{ \nu}}(\G), L^{\frac{2Q}{\nu}}(\G))$-unique.
\end{itemize}
\end{theorem}
\begin{proof}
We keep the notation used in the proof of Theorem \ref{thm.un.pos.1} and adapt the reasoning there if necessary. Observe that the next inequality follows by H\"older's inequality combined with the Sobolev embeddings \eqref{inclusions}:
\[
    \|(\tilde{V}_\epsilon-V_\epsilon)(\cdot)\tilde{u}_\epsilon(t,\cdot)\|_{L^2(\G)} \leq \|\tilde{V}_\epsilon-V_\epsilon\|_{L^{\frac{2Q}{\nu }}(\G)}\|\sqrt{\R}\tilde{u}_\epsilon(t,\cdot)\|_{L^2(\G)}\,,
    \]
    uniformly in $t\in [0,T]$. Now, if in \eqref{mink} we plug in the formula for $f_\epsilon(\sigma,\cdot)$ and the last estimate, then we get 
    \[
    \|U_\epsilon(t,\cdot)\|_{L^2(\G)} \lesssim \|u_{0,\epsilon}-\tilde{u}_{0,\epsilon}\|_{L^2(\G)}+ \|\tilde{V}_\epsilon-V_\epsilon\|_{L^{\frac{2Q}{\nu}}(\G)}\|\tilde{u}_\epsilon(t,\cdot)\|_{H^{\frac{\nu}{2}}(\G)}
    \]
    Taking into account the negligiblity and moderateness assumptions we deduce that 
    \[
    \|U_\epsilon(t,\cdot)\|_{L^2(\G)} \lesssim \omega(\epsilon)^{N}\,, \quad \text{for all}\, N \in \mathbb{N}\,,
    \]
    uniformly in $t \in [0,T]$, and for some suitable $\omega(\epsilon)$ and this finishes the proof of Theorem \ref{thm.un.pos2}.
\end{proof}

In the next theorem we prove the uniqueness of the very weak solutions to \eqref{heat.eq} in the general case of a real coefficient $V$. 

\begin{theorem}\label{thm.un.any}
Let $V\in \mathbb{R}$. Let also $u_0 \in L^2(\G)\cup \mathcal{E}'(\G)$ and $V \in L^{\infty}(\G)\cup \mathcal{E}'(\G)$. Then, the very weak solution to the Cauchy problem \eqref{heat.eq} is $(L^{\infty}(\G),L^\infty (\G))$-unique or $(L^{\infty}(\G), L^{\frac{2Q}{\nu}}(\G))$-unique, where in the second case we assume that $2Q \geq \nu$.
\end{theorem}

\begin{proof}
The arguments developed here follow the same line as in Theorems \ref{thm.un.pos.1} and \ref{thm.un.pos2}. The notation used there is adopted here as well. First we recall the estimate for the solution $U_\epsilon$ yo the Cauhcy problem \eqref{forU1}:
\[
\|U_\epsilon(t,\cdot)\|_{L^2(\G)}\leq \|W_\epsilon(t,\cdot)\|_{L^2(\G)}+\int_{0}^{T}\|\tilde{U}_\epsilon(t-\sigma,\cdot;\sigma)\|_{L^2(\G)}\,d\sigma\,.
\]
Now, the $L^2$-norms appearing on the right-hand side of the above inequality shall be estimated after applying twice the inequality \eqref{prop.neg.a.claim} for the classical solution to \eqref{heat.eq} in the case of a real potential. In particular we obtain
\begin{equation}\label{EQ:thm.un.3a}
  \|U_\epsilon(t,\cdot)\|_{L^2(\G)}\lesssim \|u_{0,\epsilon}-\tilde{u}_{0,\epsilon}\|_{L^2(\G)}\exp\left(t \|V_\epsilon\|_{L^\infty (\G)} \right)+\|f_\epsilon(\sigma,\cdot)\|_{L^2(\G)}\exp\left(t \|V_\epsilon\|_{L^\infty (\G)} \right)\,.  
\end{equation}
It is necessary for our purposes to consider the $L^\infty$-moderateness of the net $V_\epsilon$ under the choice of the function $\omega$ to be given by $\omega(\epsilon):=\left(\log \epsilon^{-N_0} \right)^{-\frac{1}{N_0}}$, where $N_0 \in \mathbb{N}$ is such that the moderateness assumption can be expressed as:
\[
\|V_\epsilon\|_{L^\infty} \lesssim \omega(\epsilon)^{-N_0}=\log (\epsilon^{-N_0})\,,
\]
uniformly in $\epsilon \in (0,1]$.
Let us plug into the estimate \eqref{EQ:thm.un.3a} the moderateness assumption as above and get 
\begin{equation}
    \label{EQ:thm.un.3b}
    \|U_\epsilon(t,\cdot)\|_{L^2(\G)}\lesssim e^{-TN_0} \left[\|u_{0,\epsilon}-\tilde{u}_{0,\epsilon}\|_{L^2(\G)}+ \|f_\epsilon(\sigma,\cdot)\|_{L^2(\G)}\right]\,,
\end{equation}
uniformly in $\epsilon \in (0,1]$. We next need to consider the negligibility of nets involved into \eqref{EQ:thm.un.3b}. First consider the case where the net $\tilde{V}_\epsilon-V_\epsilon$ is $L^\infty$-negligible. Then, using the estimate 
\[
\|f_\epsilon(\sigma,\cdot)\|_{L^2} \leq \|\tilde{V}_\epsilon-V_\epsilon\|_{L^\infty} \|\tilde{u}_\epsilon\|_{L^2}\,,
\]
and the $L^2$-moderateness of $\tilde{u}_\epsilon$ as being the very weak solution to \eqref{heat.eq} we get 
\[
\|f_\epsilon(\sigma,\cdot)\|_{L^2} \leq \omega(\epsilon)^{-N}\,,
\]
for any $N \in \mathbb{N}$. The latter combined with the $L^2$-negligibility assumption of the net $u_{0,\epsilon}-\tilde{u}_{0,\epsilon}$  under the estimate \eqref{EQ:thm.un.3b} implies the $L^2$-negligiblity of the net $u_\epsilon-\tilde{u}_\epsilon$ and we have shown that the very weak solution to the Cauchy problem \eqref{heat.eq} is $(L^{\infty}(\G),L^\infty (\G))$-unique. Now, for the case where the net $\tilde{V}_\epsilon-V_\epsilon$ is $L^{\frac{2Q}{\nu}}$-negligible, applying H\"older's inequality and using the Sobolev embedding \eqref{inclusions} we get 
\[
   \|f_\epsilon(\sigma,\cdot)\|_{L^2} \leq  \|(\tilde{V}_\epsilon-V_\epsilon)(\cdot)\tilde{u}_\epsilon(t,\cdot)\|_{L^2(\G)} \leq \|\tilde{V}_\epsilon-V_\epsilon\|_{L^{\frac{2Q}{\nu }}(\G)}\|\sqrt{\R}\tilde{u}_\epsilon(t,\cdot)\|_{L^2(\G)}\,.
    \]
    Now, since by Corollary \ref{cor.any.sob} the very weak solution $\tilde{u}_\epsilon$ is $H^{\frac{\nu}{2}}$-moderate, by the last estimate we deduce that the net $f_\epsilon$ is $L^2$-negligible, and consequently also the difference of nets $u_{0,\epsilon}-\tilde{u}_{0,\epsilon}$. The last observations implies that the very weak solution to the problem \eqref{heat.eq} is $(L^\infty,L^{\frac{2Q}{\nu}})$-unique, and the proof of Theorem \ref{thm.un.any} complete.
\end{proof}

To summarize in the general case $V \in \mathbb{R}$, the very weak solution to the Cauchy problem \eqref{heat.eq} can be $(X,Y)$-unique with $X=L^{\infty}$ and $Y=L^{\infty} \cup L^{\frac{2Q}{\nu}}$, while in the case $V\geq 0$ we allow $X=L^{\infty} \cup \{L^{\frac{Q}{\nu}}\cap L^{\frac{2Q}{\nu}} \}$.

\section{Consistency result}
 
 We conclude this paper by proving the consistency result with the classical case. This means to show that when the Cauchy problem \eqref{heat.eq} is well posed, then the very weak solution can be recaptured by the classical one in the $L^2$-sense. 
 To see this we need to assume that the approximating, to $V$, nets, as described in Definitions \ref{def.vws.pos} and \ref{def.vws.any} above on the existence of the very weak solution, are regularisations of $V$ via a Friedrichs mollifier. As in the previous sections, the cases of a non-negative and of a real potential will be treated separately.
 
 Before we engage with the proof of our claim, we recall the following space of functions:
 \begin{definition}[Space $C_0(\G)$]
 We denote by $C_0(\G)$ the space of continuous functions that vanish at infinity; i.e., we write $f \in C_0(\G)$, if for every $\delta>0$ there exists a compact set $K_\delta$ such that 
 \[
 |f|< \delta \quad \text{on}\quad \G\setminus K_\delta\,. 
 \]
 \end{definition}

 We note that the space $(C_0(\G),\|\cdot\|_{L^\infty})$ is a Banach space. The following observation is essential for our purposes:
 \begin{remark}\label{rem.Linfty}
 For a function $f \in C_0(\G)$, we know that for the corresponding regularised net $f_\epsilon$ we have: \[
 \|f_\epsilon\|_{L^\infty(\G)}\leq C<\infty\,,
 \] 
 uniformly in $\epsilon \in (0,1]$. 
 \end{remark}

\begin{theorem}[Consistency with the classical solution, case $V \geq 0$]
\label{thm.consis.pos.1}
 Let $V$ be a non-negative potential, and let $V_\epsilon$ be its regularisation. Assume also that $u_0 \in H^{\frac{ \nu}{2}}(\G)$. If one of the two following conditions is satisfied, then the regularised net $u_\epsilon$ converges, as $\epsilon \rightarrow 0$, in  $L^2(\G)$, to the classical solution $u$ given by Proposition \ref{prop.pos.b}:
 \begin{enumerate}
     \item \label{itm1} We have $Q > \nu $, and $V \in L^{\frac{2Q}{\nu }}(\G)\cap L^{\frac{Q}{\nu }}(\G) $;
     \item \label{itm2} $V \in C_0(\G)$.
 \end{enumerate}

\end{theorem}

\begin{proof}
For $u$ and for $u_\epsilon$ as in the hypothesis, we introduce the auxiliary notation $W_\epsilon(t,x):=u(t,x)-u_\epsilon(t,x)$. Then, the net $W_\epsilon$ is a solution to the Cauchy problem
	\begin{equation}\label{eq.m=0}
	\begin{cases*}
		\partial_{t}W_\epsilon(t,x)+\R W_\epsilon(t,x)+V_\epsilon(x)W_\epsilon(t,x)=f_\epsilon(t,x), \\
		W_\epsilon(0,x)=(u_0-u_{0,\epsilon})(x)\,,
	\end{cases*}
	\end{equation}
	where we denote $f_\epsilon(t,x):=(V_\epsilon(x)-V(x))u(t,x)$.
Analogously to Theorem \ref{thm.un.pos.1}, an application of Duhamel's principle gives:
\begin{eqnarray}
\label{cons.thm1.heat}
 \|W_\epsilon(t,\cdot)\|_{L^2(\G)} & \lesssim & \|u_{0}-u_{0,\epsilon}\|_{L^2(\G)}+\int_{0}^{T}\|f_\epsilon(\sigma,\cdot)\|_{L^2(\G)}\,d\sigma\nonumber\\
 & = & \|u_{0}-u_{0,\epsilon}\|_{L^2(\G)}+ \int_{0}^{T}\|(V_\epsilon-V)(\cdot)u(t,\cdot)\|\,d \sigma\,,
\end{eqnarray}
 where we have applied Minkowski's integral inequality. 
 
 Let us first consider the case \eqref{itm1}: standard arguments from functional analysis imply that the below $L^p$-norms converge to $0$ as $\epsilon \rightarrow 0$:
\[
\|u_{0}-u_{0,\epsilon}\|_{L^2(\G)}\,,\|V_\epsilon-V\|_{L^{\frac{2Q}{\nu }}(\G)}\,.
\]
The integrated quantity on the right-hand side of \eqref{cons.thm1.heat} is finite since
\[
\|(V_\epsilon-V)(\cdot)u(t,\cdot)\|\lesssim  \|V_\epsilon-V\|_{L^{\frac{2Q}{\nu }}(\G)}\|\sqrt{\R}u(\sigma,\cdot)\|_{L^2(\G)}\,,
\]
as follows by the Sobolev embeddings \eqref{emb},
and consequently, since $u \in H^{\frac{\nu}{2}}(\G)$, we also have the convergence
\[
\|(V_\epsilon-V)(\cdot)u(\sigma,\cdot)\|_{L^2(\G)} \lesssim \|V_\epsilon-V\|_{L^{\frac{2Q}{\nu }}(\G)}\|\sqrt{\R}u(\sigma,\cdot)\|_{L^2(\G)} \rightarrow 0\,.
\]
Combining  \eqref{cons.thm1.heat} with the below observations and using Lebesgue's dominated convergence theorem we get 
 \begin{equation}
     \label{duh.e.0}
      \|W_\epsilon(t,\cdot)\|_{L^2(\G)} \rightarrow 0\,,
 \end{equation}
 uniformly in $t \in [\sigma,T]$, where $\sigma$ is taken to be fixed. The last conludes the proof of the theorem for the case \eqref{itm1}.
 
 To deal with the case  \eqref{itm2}, inequality on the right-hand side of \eqref{cons.thm1.heat} shall now be estimated as:
 \[
  \|(V_\epsilon-V)(\cdot)u(\sigma,\cdot)\|_{L^2(\G)}\leq \|V_\epsilon-V\|_{L^\infty}\|u(\sigma,\cdot)\|_{L^2}\,.
 \]
 Using Lemmas 3.1.58, 3.1.59 in \cite{FR16} for $V \in C_{0}(\G)$ we have the convergence: $\|V_\epsilon-V\|_{L^\infty}\rightarrow 0$, as $\epsilon \rightarrow 0$, and the last implies in turn that convergence in \eqref{duh.e.0} is satisfied under the alternative assumptions \eqref{itm2}. Summarising the above, we have shown that the very weak solution converges to the classical one in $L^2$ in both cases \eqref{itm1} and \eqref{itm2}, and the proof of Theorem \ref{thm.consis.pos.1} is complete.
\end{proof}

\begin{theorem}\label{thm.consis.2}
Let $V$ be a real coefficient, and let $u_0 \in L^2(\G)$. Assume also that  $V \in C_0(\G)$, and also that $V_\epsilon$ is a regularisation of $V$. Then, the regularised net $(u_\epsilon)_\epsilon$ converges, as $\epsilon \rightarrow 0$, in $L^2(\G)$, to the classical solution $u$ given by Proposition \ref{prop.neg.a}. 
\end{theorem}


\begin{proof}[Proof of Theorem \ref{thm.consis.2}]
We keep the same notation as in Theorem \ref{thm.consis.pos.1}. Using the estimate \eqref{prop.neg.a.claim} and arguments similar to those developed in Theorem \ref{thm.consis.pos.1} we obtain 
\begin{eqnarray*}
\|W_\epsilon(t,\cdot)\|_{L^2(\G)} & \lesssim & \exp(t \|V_\epsilon\|_{L^\infty}) \|u_0-u_{0,\epsilon}\|_{L^2(\G)}\\
& + & \|V_\epsilon-V\|_{L^\infty}\int_{0}^{T}\exp(\sigma \|V_\epsilon\|_{L^\infty})\|u(\sigma,\cdot)\|_{L^2(\G)}\,d\sigma\,,
\end{eqnarray*}
uniformly in $t \in [0,T]$. Now, by Remark \ref{rem.Linfty} and since the convergence in $L^2$-norm and in $L^\infty$-norm of the differences of nets $u_0-u_{0,\epsilon}$ and $V_\epsilon-V$, respectively, is as we showed in Theorem \ref{thm.consis.pos.1}, we see that $u_\epsilon \rightarrow u$ in the $L^2$-sense, and the proof of Theorem \ref{thm.consis.2} is complete.
\end{proof}

\begin{remark}\label{finrem}
To be consistent with the previous work \cite{ARST21c} on the heat equation on the real case, let us point out that the assumption on the functional space $L^\infty(\G)$ of the coefficient $V$ (denoted by $q$ there) as appears in the consistency results (see Theorems 2.8 and 3.5 in  \cite{ARST21c}) should be restricted to its subspace $C_0(\G)$. Indeed, under such assumptions, the necessary $L^\infty$-convergence of the net $V_\epsilon-V$ is granted, and Theorems 2.8 and 3.5 in \cite{ARST21c} would then follow as a spacial case of Theorem \ref{thm.consis.pos.1} and \ref{thm.consis.2}, respectively, in the particular case where the group $\G$ is the trivial one and the Rockland operator is the usual Laplace operator; in symbols, when $\G=\mathbb R^d$ and $\mathcal{R}=-\Delta$. 
\end{remark}

\end{document}